\documentclass[11pt]{article}
\usepackage{a4wide}  
\usepackage{amsmath,amssymb,bbm}
\usepackage{mathrsfs}
\usepackage{mathenv}
\usepackage{fancybox}
\usepackage{version}
\usepackage[backend=bibtex,style=numeric,maxbibnames=99]{biblatex}
\usepackage{authblk}
\usepackage{empheq}
\usepackage{mathtools}
\usepackage{centernot}

\usepackage{blindtext}
\usepackage{multicol}
\usepackage{color}
\usepackage{comment}
\usepackage{wrapfig}
\usepackage{graphicx}
\usepackage[dvipsnames]{xcolor}

\usepackage{hyperref}
\hypersetup{colorlinks=true, linkcolor=blue, filecolor=magenta, urlcolor=blue}
\usepackage{amsthm,cleveref}
\providecommand{\keywords}[1]{\textbf{\textit{Keywords---}} #1}

\newcommand{\R}{\mathbb{R}}
\newcommand{\T}{\mathbb{T}}

\newcommand{\norme}[1]{\left\Vert #1\right\Vert}
\newcommand{\abs}[1]{\left\lvert #1\right\rvert}
\newcommand{\intervalleoo}[2]{\mathopen{]}#1\, ,#2\mathclose{[}}
\newcommand{\intervalleff}[2]{\mathopen{[}#1\, ,#2\mathclose{]}}
\newcommand{\intervalleof}[2]{\mathopen{]}#1\, ,#2\mathclose{]}}

\newtheorem{theorem}{THEOREM}[section]

\newtheorem{lemma}[theorem]{LEMMA}

\newtheorem{remark}[theorem]{REMARK}

\newtheorem{definition}[theorem]{DEFINITION}

\begin{document}
	
	%!TeX spellcheck = en_US	
	
	\title{Stability estimates for magnetized Vlasov equations}
	
	\author{Alexandre Rege\footnote{ETH Z\"urich, Department of Mathematics, R\"amistrasse 101, 8092 Z\"urich Switzerland.\\
			Email: \href{mailto:alexandre.rege@math.ethz.ch}{alexandre.rege@math.ethz.ch}}}
	
	\maketitle
	
	\begin{abstract}
%We show stability of solutions to the magnetized Vlasov--Poisson system with a non-uniform magnetic field.  %
%For the Vlasov equation with a uniform magnetic field, we generalize Iacobelli's improved Dobrushin stability estimate \cite[Theorem 2.1]{I22}.
We present two results related to magnetized Vlasov equations. Our first contribution concerns the stability of solutions to the magnetized Vlasov--Poisson system with a non-uniform magnetic field using the optimal transport approach introduced by Loeper \cite{L06}. We show that the extra magnetized terms can be suitably controlled by imposing stronger decay in velocity on one of the distribution functions, illustrating how the external magnetic field creates anisotropy in the evolution. This allows us to generalize the classical $2$-Wasserstein stability estimate by Loeper \cite[Theorem 1.2]{L06} and the recent stability estimate using a kinetic Wasserstein distance by Iacobelli \cite[Theorem 3.1]{I22} to the magnetized Vlasov--Poisson system. In our second result, we extend the improved Dobrushin estimate by Iacobelli \cite[Theorem 2.1]{I22} to the magnetized Vlasov equation with a uniform magnetic field.

	\end{abstract}

%\textcolor{blue}{I've explained the problem at page 5. It concerns section 1.1 where I try to generalize your paper with Jonathan for the magnetized Vlasov--Poisson system with $B$ depending on the position. This problem comes from the analysis starting bottom of page 3 and ending at the middle of page 5, where I try to bound the extra term in the Wasserstein distance coming from the external magnetic field $B$.\\
%Concerning the rest of the document, section 1.2 isn't finished and section 2 is finished (the math at least).}

%\section{A kinetic Wasserstein distance for the magnetized Vlasov--Poisson system}

\section{Introduction}

The first celebrated stability estimate in kinetic theory is due to Dobrushin, where in \cite{D79} the existence, uniqueness, and stability of solutions to the Vlasov equation with $C^{1,1}$ potential was established using a fixed point argument involving the $1$-Wasserstein distance, see \cite[Chapter 1.4]{G16} for a complete exposition of Dobrushin's estimate. This argument has been widely used to treat similar questions for various kinetic equations, for example in \cite{CCR11,CCS19,FR18,GMR13,GI21}. More recently, Dobrushin's estimate was refined by Iacobelli in \cite{I22} to take into account the small interaction and small time regimes.

Another important stability estimate is due to Loeper in \cite{L06}, where the uniqueness and stability of bounded density solutions to the Vlasov--Poisson system in the full space was established using an optimal transport argument involving the $2$-Wasserstein distance. Loeper also showed that the same argument can be used for the 2D incompressible Euler equation to reprove Yudovich's well known uniqueness result \cite{Y63}. This argument has been adapted and replicated in numerous subsequent works, see \cite{BLR11,CR10,HKMMM20,L06bis} for some applications. For the Vlasov--Poisson system specifically, Loeper's argument was adapted on the torus by Han-Kwan and Iacobelli in \cite{HKI17}, and the uniqueness criterion and stability estimate were refined by Holding and Miot in \cite{M16,HM18} with the weakened assumption that the density belonged to an Orlicz space (see \cref{sec:VPwBmain} for more details). Interestingly, the analysis in \cite{M16,HM18} is carried out using the $1$-Wasserstein distance which allowed the authors to avoid using optimal transport techniques. Then in \cite{I22}, Iacobelli improved Loeper's stability estimate by  introducing a nonlinear quantity controlling the $2$-Wasserstein between two solutions.

In this paper, we are going focus on the stability of solutions to the Vlasov equation and Vlasov--Poisson system with a given external magnetic field, and henceforth we will denote these equations as magnetized Vlasov and magnetized Vlasov--Poisson. For the moment, the only work done on this subject can be found in \cite{R23}, where the author adapted Miot's approach in \cite{M16} to find a uniqueness criterion for the magnetized Vlasov--Poisson system with non-constant and non-uniform magnetic field. 

In the present work, we are first going to extend Loeper's \cite[theorem 1.2]{L06} and Iacobelli's \cite[theorem 3.1]{I22} stability estimates to the magnetized Vlasov--Poisson system with a non-constant (depending on time) and non-uniform (depending on the position) magnetic field. We will see that this problem presents two difficulties: the first is the non-uniform nature of the magnetic field which creates an extra nonlinear term in the differential equation verified by the characteristic flow. This is an issue because the works \cite{L06,I22} use the Lagrangian point of view, which requires a fine study of the characteristic flow or particle trajectories in the phase space. This nonlinear term will in fact create anisotropy in the dynamics because extra decay in velocity will have to be imposed on only one of the initial datum. The second difficulty is working with the $2$-Wasserstein metric, indeed we will see that in this framework the extra magnetized terms will be of the same order as those treated in \cite{L06,I22} coming from the electric field, whereas in the $1$-Wasserstein setting in \cite{R23} these magnetized terms were less singular. Indeed, as suggested by the fact that optimal transport arguments are required when dealing with $p$-Wasserstein distances with $p>1$ \cite{L06,I22,IJ22}, it seems that working with these distances creates a difficulty compared to the $1$-Wasserstein framework.

The second aim of this paper is to extend Iacobelli's improved Dobrushin estimate \cite[theorem 2.1]{I22} to the magnetized Vlasov equation with constant and uniform magnetic field. The central idea behind \cite[theorem 2.1]{I22} is to consider a quantity that encodes the evolution of the Vlasov equation renormalized by the linear dynamics, in order to obtain an estimate that is finer in the small interaction and small time regime. Here we use the same idea in the magnetized framework, with the main difference being that the situation in dimension 2 differs from the situation in dimension 3. This is due to the fact that for the two-dimensional linear dynamics the external magnetic field imposes rotation of particles in the plane, whereas in dimension 3 there is still dispersion parallel to the magnetic field.

Before giving the outline of the paper, we recall the definition of Wasserstein distances on the phase space $\mathbb{T}^d \times \R^d$, which we will use throughout this work.

\begin{definition}
Let $\mu,\nu$ be two probability measures on $\mathbb{T}^d \times \R^d$. The Wasserstein distance of order $p$ or $p$-Wasserstein distance, with $p\geq 1$, between $\mu$ and $\nu$ is defined as
\begin{equation*}
W_p(\mu,\nu) \coloneqq \left(\underset{\pi \in \Pi(\mu,\nu)}{\inf} \int_{(\mathbb{T}^d \times \R^d)^2} (\abs{x-y}^p+\abs{v-w}^p)d\pi(x,v,y,w)\right)^{1/p},
\end{equation*}
where $\Pi(\mu,\nu)$ is the set of all probability measures on $(\mathbb{T}^d \times \R^d)^2$ with marginals $\mu$ and $\nu$, which are also called couplings. A coupling is said to be optimal if it minimizes the Wasserstein distance.
\end{definition}

\textbf{Outline of the paper:} This paper is organized as follows. 

In \cref{sec:VP} we will show stability of solutions to the magnetized Vlasov--Poisson system with a non-uniform non-constant magnetic field verifying \eqref{reguB}. We will first present the main result and some important related remarks in \cref{sec:VPwBmain}, then after some preliminaries in \cref{sec:VPwBprel}, we will give the proof in \cref{sec:VPwBproof}. Lastly, in \cref{sec:Blogl}, we detail how to treat the case of a log-lipschitz magnetic field. This section extends the results from \cite{L06,I22} to the magnetized case and, as a corollary, the uniqueness result in \cite{R23} to the $2$-Wasserstein metric. 

In \cref{sec:V}, we present and prove an improved Dobrushin estimate for the magnetized Vlasov equation with a constant and uniform magnetic field extending the idea from \cite[Section 2]{I22} to the magnetized framework. We start by giving the main result in \cref{sec:VwBmain} and finally give the proof in \cref{sec:VwBproof}.

\section{Stability for the magnetized Vlasov--Poisson system}\label{sec:VP}

We study the magnetized Vlasov--Poisson system on the torus in dimension $d=2,3$, which is given by the following set of equations:

\begin{equation}\label{sys:VPwB}\tag{VPB}
\left\{
\begin{aligned}
& \partial_t f + v\cdot \nabla_x f + \left(E+v \wedge B\right) \cdot \nabla_{v}	 f= 0, \\
& E=-\nabla_x U,\\
& \Delta_x U=1-\rho,\\
& f(0,x,v)=f^{in}(x,v)\geq 0.
\end{aligned}
\right.
\end{equation}
where $f^{in} \coloneqq f^{in}(x,v) \in \R^+$ is a probability measure, $f:=f(t,x,v) \in \R^+$ is the distribution function of particles at time $t \in \R^+$, position $x\in \mathbb{T}^d$ and velocity $v \in \R^d$, $\rho(t,x):=\int_{\R^d} f(t,x,v)dv$ is the charge particle density, $U \coloneqq U(t,x),E \coloneqq E(t,x)$ are respectively the self-induced electrostatic potential and electric field, and $B\coloneqq B(t,x)$ is the external magnetic field. In the case of $d=2$, the magnetic field will be given by $B\coloneqq (0,0,b(t,x))$ so that the electron dynamics remain $2$-dimensional. Since the Vlasov equation in \eqref{sys:VPwB} is a conservative transport equation, $f(t)$ will also be a probability measure for all $t \in \R^+$. This system models the evolution of a set of charged particles interacting through the Coulomb law \emph{and subject to an external magnetic field $B$}, making it particularly relevant for studying various physical systems, most notably plasmas. 

%The external magnetic field $B:=B(t,x)$ will verify
%\begin{equation}\label{reguB}
%B \in L^\infty_{loc}\left(\R^+, \underset{0<\alpha< 1}{\bigcap} C^{0,\alpha}(\T^d)\right),
%\end{equation}
%where $L^\infty_{loc}\left(\R^+, \underset{0<\alpha< 1}{\bigcap} C^{0,\alpha}(\T^d)\right)$ is the space of functions locally bounded in time and H\"older continuous for all $0<\alpha < 1$ in space with H\"older norm uniformly bounded in $\alpha$ and in time.
%
%This regularity assumption is quite reasonable because it is verified for example by functions that are log-lipschitz in position and locally bounded in time. In particular, since we're working on the torus, \eqref{reguB} implies that for all $T>0$
%\begin{equation}\label{Bbounded}
%B \in L^\infty(\intervalleff{0}{T} \times \T^d).
%\end{equation}
%In order to lighten the inequalities involving the magnetic field in the proof, we use the simplified notation $\norme{B}_{\infty}\coloneqq \norme{B}_{L^\infty(\intervalleff{0}{T}\times \T^d)}$ and $\norme{B}_{L^\infty\left( W^{1,\infty}\right)}=\norme{B}_{L^\infty\left(\intervalleff{0}{T}, \underset{0<\alpha< 1}{\bigcap} C^{0,\alpha}(\T^d)\right)}$.

For this section's main result, we will assume that the external magnetic field is bounded in time and Lipschitz in position
\begin{equation}\label{reguB}
B \in L^\infty_{loc}\left(\R^+, W^{1,\infty}(\mathbb{T}^d)\right).
\end{equation}
In particular this means that $B$ will be H\"older continuous for all $\alpha \in \intervalleoo{0}{1}$. Indeed for $T>0$ we have for all $t \in \intervalleff{0}{T}$and $x,y \in \T^d$  
\begin{equation}\label{Bholder}
	\abs{B(t,x)-B(t,y)} \leq \norme{B}_{L^\infty\left(\intervalleff{0}{T}, W^{1,\infty}(\mathbb{T}^d)\right)}\abs{x-y}^\alpha.
\end{equation}
We will also discuss how to obtain stability estimates for magnetic fields that are just log-lipschitz in position (see \cref{rem:Bloglip} and \cref{sec:Blogl}).

In order to lighten the inequalities involving the magnetic field in the proof, we use the shortened notation $\norme{B}_{\infty}\coloneqq \norme{B}_{L^\infty(\intervalleff{0}{T}\times \T^d)}$ and $\norme{B}_{L^\infty\left(\intervalleff{0}{T}, W^{1,\infty}(\mathbb{T}^d)\right)}=\norme{B}_{L^\infty\left( W^{1,\infty}\right)}$.

The literature concerning the existence of solutions to the Vlasov--Poisson system (\eqref{sys:VPwB} when $B=0$) is considerable. In the full space setting, we mention \cite{AR75} where the existence of weak solutions was shown, \cite{P92} where the existence of smooth solutions was proved, and \cite{LP91} where propagation of velocity moments for weak solutions to Vlasov--Poisson was shown, which also implies existence of smooth solutions. The case of the torus was solved by Batt and Rein who showed existence of smooth solutions in \cite{BR91}, and was subsequently improved by Pallard in \cite{P12}. The case $B \ne 0$ was recently treated by the author who proved propagation of velocity moments for weak solutions to \eqref{sys:VPwB} in the full space in the case of a constant and uniform magnetic field in \cite{R21}, and in both the full space and torus in the case of non-constant uniform magnetic field in \cite{R23}, extending the results from \cite{LP91,P12} to the magnetized framework. In the case of a non-uniform magnetic field, existence of weak solutions to \eqref{sys:VPwB} is a corollary of the work by DiPerna and Lions \cite{DLvlasovmax} on the existence of renormalized solutions for the Vlasov--Maxwell system. However, the existence of smooth solutions with a non-uniform $B$ is still open.

\subsection{Main result}\label{sec:VPwBmain}

We are going to prove the following result:
\begin{theorem}\label{theo:main}
	Let $f_1,f_2$ be two weak solutions of \eqref{sys:VPwB} with respective densities $\rho_1$ and $\rho_2$. We define the function

	\begin{equation*}\label{def:A}
	A(t) \coloneqq \norme{\rho_1(t)}_{L^\infty(\T^d)}+\norme{\rho_2(t)}_{L^\infty(\T^d)}.
	\end{equation*}
	We also define 
	\begin{equation*}
	J(s)=A(s)+\norme{B}_{L^\infty\left( W^{1,\infty}\right)}\left( e^{s\norme{ B}_\infty} +\int_{0}^{s}\left(1+\norme{\rho_2(u)}_{L^\infty(\mathbb{T}^d)}\right)e^{(s-u)\norme{B}_{\infty}}du\right).
	\end{equation*}
	For $T>0$, assume that $B$ verifies \eqref{reguB} and that $A$ satisfies,
	\begin{equation}\label{condi:A}
	A \in L^1(\intervalleff{0}{T}).
	\end{equation}
	Assume also that there exists a universal constant $C_0$ such that for all $k\geq 1$
	\begin{equation}\label{condi:Orli}
	\int_{\mathbb{T}^d \times \R^d}\abs{v}^k df_2^0(x,v) \leq (C_0 k)^k.
	\end{equation}
	Then we can write the two following statements:
	\begin{enumerate}
		\item There exists a dimensional constant $c_d$ such that if $W_2^2(f_1(0),f_2(0))$ is sufficiently small so that $W_{2}^{2}(f_1(0),f_2(0)) < 1/e^2$ and
		\begin{equation}\label{condi:W2ini}
		\abs{\log\left(W_2^2(f_1(0),f_2(0))\right)} \geq \exp\left(c_d \int_{0}^{T} J(s)ds\right),
		\end{equation}
		then
		\begin{equation}
		W_2^2(f_1(t),f_2(t)) \leq
		\exp\left[\log\left(W_2^2(f_1(0),f_2(0))\right)\exp\left(-c_d \int_{0}^{t} J(s) ds\right)\right].
		\end{equation}
		\item There exists a dimensional constant $C_d$ and a universal constant $c_0$ such that if $W_2^2(f_1(0),f_2(0))$ is sufficiently small so that $W_2^2(f_1(0),f_2(0)) < c_0$ and
		\begin{equation}\label{condi:W2inibis}
		\sqrt{\abs{\log\left(W_2^2(f_1(0),f_2(0))\abs{\log\left(\frac{1}{2}W_2^2(f_1(0),f_2(0))\right)}\right)} } \geq C_d \int_{0}^{T} J(s)ds+1,
		\end{equation}
		then
		\begin{equation}
		\begin{aligned}
		&W_2^2(f_1(t),f_2(t)) \leq\\
		&2\exp\left[-\left(\sqrt{\abs{\log\left(W_2^2(f_1(0),f_2(0))\abs{\log\left(\frac{1}{2}W_2^2(f_1(0),f_2(0))\right)}\right)}}-C_d \int_{0}^{t} J(s)ds\right)^2\right].
		\end{aligned}
		\end{equation}
	\end{enumerate}
	
\end{theorem}
We see that if $B$ verifies \eqref{reguB} and with the additional assumption \eqref{condi:Orli} (compared to the unmagnetized setting in \cite{L06,I22}), we obtain the same stability estimate  for the magnetized Vlasov--Poisson system as Loeper in \cite{L06}, which corresponds to statement 1 in \cref{theo:main}, and the same stability estimate as Iacobelli in \cite{I22}, which corresponds to statement 2 in \cref{theo:main}.

Now we give a number of important remarks related to \cref{theo:main}.
\begin{remark}
	We see how a non-uniform external magnetic field creates anisotropy in the stability of solutions to \eqref{sys:VPwB}, because more regularity needs to be imposed but only on one of the solutions. This is because the conditions \eqref{condi:A} and \eqref{condi:Orli} are very different. Indeed, solutions verifying \eqref{condi:Orli} with unbounded macroscopic density $\rho$ were found in \cite{M16}. Conversely, if $f^{in}(x,v)\coloneqq \frac{g(x)}{1+\abs{v}^p}$ with $g \in L^\infty(\T^d)$ and $p>d$, then the unique solution to \eqref{sys:VPwB} verifies \eqref{condi:A} but not \eqref{condi:Orli}.
\end{remark}
\begin{remark}
	For $\gamma\geq 1$, similarly to what is done in \cite{HM18} we can define the exponential Orlicz space $\left(L_{\phi_\gamma}(\mathbb{T}^d\times \R^d,df_2^0),\norme{\cdot}_{\phi_\gamma}\right)$ where $\phi_\gamma$ is the $N$-function given by $\phi_\gamma(x)=\exp(x^\gamma)-1$, and the norm $\norme{\cdot}_{\phi_\gamma}$ on $L_{\phi_\gamma}(\mathbb{T}^d\times \R^d,df_2^0)$ is given by
	\begin{equation*}\label{normeOrli}
	\norme{g}_{\phi_\gamma}\coloneqq \underset{r \geq \gamma}{\sup} \, r^{-\frac{1}{\gamma}}\norme{g}_{L^r(\mathbb{T}^d\times \R^d,df_2^0)}.
	\end{equation*}
	We refer the reader to \cite{HM18} for the precise definitions of Orlicz space and $N$-function. Using these objects, we notice that \eqref{condi:Orli} is equivalent to saying that the function $h:(x,v) \in \mathbb{T}^d\times \R^d  \mapsto \abs{v}$ is in the exponential Orlicz space $L_{\phi_1}(\mathbb{T}^d\times \R^d,df_2^0)$. More generally instead of \eqref{condi:Orli} we could have imposed that \begin{equation}\label{ineq:decayvelo}
	\int_{(\mathbb{T}^d \times \R^d)^2}\abs{v}^k df_2^0(x,v) \leq (C_0 k)^\frac{k}{\gamma}
	\end{equation}
	for a general $\gamma \geq 1$ which is equivalent to saying that $h$ defined as above is in $L_{\phi_\gamma}(\mathbb{T}^d\times \R^d,df_2^0)$. The case $\gamma=1$ is of course the weakest assumption, which is why we choose to impose this condition on $f_2(0)$ in \cref{theo:main}. Furthermore, we will see in the proof of \cref{theo:main} that the magnetized terms are of the same order as the other terms coming from the electric field. If we were to impose \eqref{ineq:decayvelo} with $\gamma > 1$, this extra regularity on $f_2(0)$ would allow us to bound the magnetized terms with the electric terms. Finally, \textbf{we note that Maxwellians verify \eqref{condi:Orli}}.

\end{remark}
\begin{remark}\label{rem:Bloglip}
	If we assume that the magnetic field is just log-lipschitz in position, then we can obtain stability estimates similar to those in \cref{theo:main} under the crucial assumption that $f_2(0)$ verifies \eqref{ineq:decayvelo} for $\gamma =2$, which as explained above is stronger than the condition \eqref{condi:Orli} imposed in \cref{theo:main}. What appears is that stronger decay on the velocity moments of $f_2(0)$ is required to treat less regular magnetic fields. %Indeed, in this case the proof of \cref{theo:main} can be quite easily adapted (see \cref{sec:Blogl}) by strengthening the assumption on the decay of the velocity moments of $f_2(0)$ (or $f_2(0)$). 
	In fact, by imposing \eqref{ineq:decayvelo} for larger and larger $\gamma$, we can even consider magnetic fields that are log$^{\, \beta}$-lipschitz in position for $\beta \in \intervalleoo{1}{\frac{3}{2}}$, by which we just mean that for $t \in \intervalleff{0}{T}, x,y \in \T^d$
	\begin{equation}\label{ineq:Blogl}
		\abs{B(t,x)-B(t,y)} \leq C \abs{x-y}\abs{\log\abs{x-y}}^\beta,
	\end{equation}
	with $C>0$ independent of $x,y,t$. In \cref{sec:Blogl}, we write the precise estimates obtained when $B$ is log-lipschitz (verifies \eqref{ineq:Blogl} with $\beta=1$), and explain how the proof of \cref{theo:main} can be adapted to this situation. The general case $\beta \in \intervalleoo{1}{\frac{3}{2}}$ can be easily deduced from the ideas in \cref{sec:Blogl}.
\end{remark}
\begin{remark}
	We can obtain stability estimates for the magnetized 2D incompressible Euler equation using the same approach as Loeper in \cite{L06}. See \cite{BMN97,GSR05} where similar magnetized fluid models are studied. This system is given by
	\begin{equation}\label{sys:EwB}
	\left\{
	\begin{aligned}
	& \partial_t u+(u\cdot \nabla_x)u=-\nabla p+ u \wedge B, \\
	&  \mbox{div}\, u=0,\\
	& u(0,x)=u^{in}(x),
	\end{aligned}
	\right.
	\end{equation}
	where $u \coloneqq u(t,x) \in \R^2$ is the velocity of particles, and $p\coloneqq p(t,x) \in \R$ is the internal pressure at time $t \in \R^+$ and position $x \in \T^2$ or $\in \R^2$. We also recall that $B=(0,0,b(t,x))$ in this two dimensional setting. Like for the classical 2D Euler equation we can rewrite \eqref{sys:EwB} in vorticity formulation, exploiting the fact that $u,B$ are divergence free and that $B=(0,0,b(t,x))$ to obtain
	\begin{equation}\label{sys:EwB2}
	\left\{
	\begin{aligned}
	& \partial_t \omega+u\cdot \nabla_x \omega=-u \cdot \nabla_x b, \\
	&  u=\mathcal{K}\ast \omega,\\
	& \omega(0,x)=\omega^{in}(x),
	\end{aligned}
	\right.
	\end{equation}
	with $\omega=\mbox{curl}\,u$ the vorticity and $\mathcal{K}\coloneqq \frac{1}{2\pi }(\frac{-x_2}{\abs{x^2}},\frac{x_1}{\abs{x^2}})$ the classical kernel associated to the Biot--Savart law. Now we notice that the quantity $\widetilde{\omega}=\omega+b$ verifies the transport equation $\partial_t \widetilde{\omega}+u \cdot \nabla_x\widetilde{\omega}=\partial_t b$. From this observation, if we assume $b \in W^{1,\infty}(\R^+ \times \T^2)$, one can easily adapt Loeper's analysis in \cite[section 4]{L06} to obtain a stability estimate for \eqref{sys:EwB2} of the same type.
\end{remark}
\begin{remark}
	This last remark concerns three important generalizations of \cref{theo:main}.
	\begin{itemize}
		\item We emphasize that similar stability estimates can be obtained in the full space setting. We chose here to focus on the periodic case, like in \cite{I22}, because these stability estimates can be used to study the quasineutral limit of \eqref{sys:VPwB} in the spirit of \cite{HKI17,GPI20,GPI23,I22}.
		\item The proof of \cref{theo:main} is conducted using the $2$-Wasserstein metric, but can be generalized to any $p$-Wasserstein distance for $1<p<+\infty$ following the recent work by Iacobelli and Junn\'e \cite{IJ22} mentioned in the introduction. However, the control of the magnetized terms in the proof of \cref{theo:main} can be easily generalized in the $p$-Wasserstein framework, so for simplification we only consider the $2$-Wasserstein case.
		\item We can combine our approach to bound the extra magnetized terms with the analysis from \cite{GPI20,GPI23} to obtain stability estimates for the magnetized Vlasov--Poisson system for ions. This system is similar to the electron model \eqref{sys:VPwB}, with the main difference being that the electric field is induced by the ions and a background of thermalized electrons distributed according to a Maxwell--Boltzmann law, meaning that one has to work with the nonlinear Poisson equation $\Delta U= \exp(U)-\rho$ instead of the standard Poisson equation.
	\end{itemize}
\end{remark}

\subsection{Preliminaries and notations}\label{sec:VPwBprel}
We begin by writing the equations satisfied by the characteristics of \eqref{sys:VPwB}:

\begin{equation}\label{def:chara}
\left\{
\begin{aligned}
& \dot{X}(s;t,x,v)=V(s;t,x,v),\\
& \dot{V}(s;t,x,v)=E(s,X(s;t,x,v))+V(s;t,x,v) \wedge B(s,X(s;t,x,v)),
\end{aligned}
\right.
\end{equation}
with 
\begin{equation}\label{def:inichara}
\left(X(t;t,x,v),V(t;t,x,v)\right)=(x,v).
\end{equation}
In the following computations, we are only going to consider characteristics with initial conditions at $t=0$ so we will write $Y(s)$ or $Y(s,x,v)$ for $Y(s;0  ,x,v)$.
%\section{The stability estimate}

We begin this section by recalling the following bound on the electric field.
\begin{lemma}
Let $T>0$, and let $f$ be a weak solution of \eqref{sys:VPwB} such that $\rho \in L^1(\intervalleff{0}{T};L^p(\mathbb{T}^d))$ with $p>d$. Then $E \in L^1(\intervalleff{0}{T};L^\infty(\mathbb{T}^d))$ and there exists $C>0$ such that for all $t \in \intervalleff{0}{T}$
\begin{equation}\label{ineq:Erho}
\norme{E(t)}_{L^\infty( \T^d)} \leq C\left(1+\norme{\rho(t)}_{L^p(\mathbb{T}^d)}\right).
\end{equation}
\end{lemma}
\begin{proof}
According to \cite{BR91}, the Green function for the Laplacian in the periodic case $\mathcal{G}$ can be written as
\begin{equation*}
\mathcal{G}(x)=G(x)+\mathcal{G}_0(x),
\end{equation*}
with $G(x)=-\frac{1}{2\pi}\log \abs{x}$ for $d=2$, $G(x)=\frac{1}{4\pi}\frac{1}{\abs{x}}$ for $d=3$ the Coulomb potential, and $\mathcal{G}_0 \in C^\infty(\mathbb{T}^d \cup \{(0,0,0)\})$.

Since the electric field verifies $E=-\nabla_x \mathcal{G} \ast (\rho-1)$, we have
\begin{align*}
\norme{E(t)}_{L^\infty(\T^d)}&\leq \left(\norme{\mathbf{1}_{\abs{x}\geq 1} \nabla_x G}_{L^\infty(\T^d)}+\norme{\nabla_x\mathcal{G}_0}_{L^\infty(\T^d)}\right) \norme{\rho(t)-1}_{L^1(\T^d)} \\
& +\norme{\mathbf{1}_{\abs{x}< 1}\nabla_x G}_{L^{p'}(\T^d)} \norme{\rho(t)-1}_{L^p(\T^d)},
\end{align*}
which immediately implies \eqref{ineq:Erho} because $\nabla_x G=C_d\frac{x}{\abs{x}^d} \in L^q(\T^d)$ for $q < \frac{d}{d-1}$, with $C$ just a numerical constant in \eqref{ineq:Erho} since $f$ is a probability measure.
\end{proof}
We also give the following estimate on the velocity characteristic $V$ when the electric field is in $L^\infty$.
\begin{lemma} Let $T>0$, and assume that $E \in L^1(\intervalleff{0}{T};L^\infty(\mathbb{T}^d))$, and $B \in L^\infty(\intervalleff{0}{T}\times \T^d)$. Then the velocity characteristic $V$ defined in \eqref{def:chara}, \eqref{def:inichara} verifies
\begin{equation}\label{ineq:V(t)}
\abs{V(t,x,v)} \leq \abs{v}\exp(t\norme{B}_{L^\infty(\intervalleff{0}{T}\times \T^d)})+\int_{0}^{t}\norme{E(s)}_{L^\infty(\mathbb{T}^d)}\exp\left((t-s)\norme{B}_{L^\infty(\intervalleff{0}{T}\times \T^d)}\right)ds,
\end{equation}
for all $t \in \intervalleff{0}{T}$.
\end{lemma} 
\begin{proof}
Thanks to \eqref{def:chara}, \eqref{def:inichara} we can write for all $t \in \intervalleff{0}{T}$,
\begin{align*}
\abs{V(t,x,v)} & \leq \abs{v}+\int_{0}^{t}\abs{E(s,X(s,x,v))}ds+\int_{0}^{t}\abs{V(s,x,v)}\abs{B(s,X(s,x,v))}ds,\\
& \leq \abs{v}+\int_{0}^{t}\norme{E(s)}_{L^\infty(\mathbb{T}^d)}ds+\norme{B}_{L^\infty(\intervalleff{0}{T}\times \T^d)}\int_{0}^{t}\abs{V(s,x,v)}ds.
\end{align*} 
This classical Gr\"onwall inequality yields the inequality
\begin{align*}
\abs{V(t,x,v)} &\leq \abs{v}+\int_{0}^{t}\norme{E(s)}_{L^\infty(\mathbb{T}^d)}ds\\
&+\int_{0}^{t}(\abs{v}+\int_{0}^{s}\norme{E(u)}_{L^\infty(\mathbb{T}^d)}du)\norme{B}_{L^\infty(\intervalleff{0}{T}\times \T^d)}\exp\left((t-s)\norme{B}_{L^\infty(\intervalleff{0}{T}\times \T^d)}\right)ds.
\end{align*}
Then we make the following basic observations 
\begin{align*}
\int_{0}^{t}\abs{v}\norme{B}_{L^\infty(\intervalleff{0}{T}\times \T^d)}\exp((t-s)\norme{B}_{L^\infty(\intervalleff{0}{T}\times \T^d)})ds=\abs{v}\exp\left(t\norme{B}_{L^\infty(\intervalleff{0}{T}\times \T^d)}\right)-\abs{v},
\end{align*}
and
\begin{align*}
&\int_{s=0}^{t}\int_{u=0}^{s}\norme{E(u)}_{L^\infty(\mathbb{T}^d)}\norme{B}_{L^\infty(\intervalleff{0}{T}\times \T^d)}\exp\left((t-s)\norme{B}_{L^\infty(\intervalleff{0}{T}\times \T^d)}\right)duds\\
&=\int_{u=0}^{t}\int_{s=u}^{t}\norme{E(u)}_{L^\infty(\mathbb{T}^d)}\norme{B}_{L^\infty(\intervalleff{0}{T}\times \T^d)}\exp\left((t-s)\norme{B}_{L^\infty(\intervalleff{0}{T}\times \T^d)}\right)duds\\
&=\int_{0}^{t}\norme{E(u)}_{L^\infty(\mathbb{T}^d)}\left(\exp\left((t-u)\norme{B}_{L^\infty(\intervalleff{0}{T}\times \T^d)}\right)-1\right)du.
\end{align*}
From these calculations we immediately deduce \eqref{ineq:V(t)}.
\end{proof}

\subsection{Proof of \cref{theo:main}}\label{sec:VPwBproof}

Throughout the proof $C$ will refer to a numerical constant that can change from one line to the next.
We consider two solutions $f_1,f_2$ associated to the two initial probability distributions $f_1(0),f_2(0)$. As usual, we write the corresponding densities, electric fields, and characteristics $\rho_1,\rho_2$, $E_1,E_2$, and $Y_1(t,x,v), Y_2(t,y,w)=(X_1(t,x,v),V_1(t,x,v)),(X_2(t,y,w),V_2(t,y,w))$. %To simplify the presentation, in most of the computations we will write $Y_i(t)$ for the characteristics. 

%We will also assume that \eqref{condi:Orli} is verified for $i=2$.

We consider $\pi_0$ an optimal $W_2$-coupling between $f_1(0)$ and $f_2(0)$ and we define the quantity $Q(t)$ defined as the unique constant (assuming it exists) such that 
\begin{equation*}
Q(t)=\frac{1}{2} \int_{(\mathbb{T}^d \times \R^d)^2} \left[\lambda(t)\abs{X_1(t,x,v)-X_2(t,y,w)}^2+\abs{V_1(t,x,v)-V_2(t,y,w)}^2\right]d\pi_0(x,v,y,w).
\end{equation*}
For $\lambda(t)=1$, this is the classical quantity introduced by Loeper in \cite{L06} that controls the $2$-Wasserstein distance between $f_1(t)$ and $f_2(t)$.
The extra weight $\lambda(t)$ was introduced in \cite{I22} and is central in the proof of the improved stability estimate for the Vlasov--Poisson system. This improvement is achieved by making the weight depend on $Q(t)$ itself through the identity $\lambda(t)=\abs{\log Q(t)}$. Here we will consider both the case $\lambda(t)=1$ to prove statement 1 in \cref{theo:main} and $\lambda(t)=\abs{\log Q(t)}$ to prove statement 2 in \cref{theo:main}.

In the following computations involving $Q$, we will often omit to write the variables $t,x,v,y,w$ to lighten the presentation, so if not stated otherwise $X_1$ will always be a function of $t,x,v$ and $X_2$ a function of $t,y,w$.

Now we differentiate $Q$ and write $Q'$ by splitting the magnetic part of the Lorentz force $V \wedge B$ in the following way
\begin{align*}
Q'(t) & =\frac{1}{2} \int_{(\mathbb{T}^d \times \R^d)^2} \lambda'(t)\abs{X_1-X_2}^2 d\pi_0\\ %=\int_{\R^6} f^{in}(x,v)(Y_1(t)-Y_2(t)) \cdot \partial_t(Y_1(t)-Y_2(t))dxdv,\\
& + \lambda(t)\int_{(\mathbb{T}^d \times \R^d)^2} (X_1-X_2) \cdot  (V_1-V_2)d\pi_0\\
& + \int_{(\mathbb{T}^d \times \R^d)^2} (V_1-V_2) \cdot (E_1(t,X_1)-E_2(t,X_2))d\pi_0\\
& + \int_{(\mathbb{T}^d \times \R^d)^2} (V_1-V_2) \cdot  \left[V_2 \wedge (B_1(t,X_1)-B_2(t,X_2)\right]d\pi_0\\
& + \int_{(\mathbb{T}^d \times \R^d)^2} (V_1-V_2) \cdot \left[(V_1-V_2) \wedge B(t,X_1)\right] d\pi_0.
\end{align*}
First, we notice that the last term is null. We will use the analysis from \cite{L06,I22} to bound the first three terms that don't depend on the magnetic field. Hence, we will only focus on controlling the penultimate term which we denote as $P(t)$. We begin by using the estimate on the velocity characteristic \eqref{ineq:V(t)} to obtain% and the estimate on the electric field \eqref{ineq:Erho} which we can use thanks to the regularity of $A$. %and then the Cauchy--Schwarz inequality.
\begin{align*}
P(t) &\leq \int_{(\mathbb{T}^d \times \R^d)^2}  \abs{V_2}\abs{V_1-V_2} \abs{B_1(t,X_1)-B_2(t,X_2)}d\pi_0,\\
&\leq \int_{(\mathbb{T}^d \times \R^d)^2}   \abs{w}e^{t\norme{ B}_\infty}\abs{V_1-V_2}\abs{B_1(t,X_1)-B_2(t,X_2)}d\pi_0\\
&+ \int_{(\mathbb{T}^d \times \R^d)^2} \left(\int_{0}^{t}\norme{E_2(s)}_{L^\infty(\mathbb{T}^d)}e^{(t-s)\norme{B}_{\infty}}
ds\right)\abs{V_1-V_2}\abs{B_1(t,X_1)-B_2(t,X_2)}d\pi_0\\
&= R(t)+S(t).
\end{align*}

%Since $\norme{\rho_{i}}_\infty \leq +\infty$ we can again bound $\norme{E_{i}}_\infty$ thanks to \eqref{ineq:Erho}:
%\begin{equation}
%\norme{E_{i}}_\infty \leq C(1+\norme{\rho_i}_\infty) =: C_
%\end{equation}
%with $i=1,2$.

We bound $S(t)$ using \eqref{reguB} for $\alpha=1-\frac{1}{p}$ with $p>1$, and the estimate on the electric field \eqref{ineq:Erho} which we can use thanks to the regularity on $A$ \eqref{condi:A}, and the Cauchy--Schwarz inequality %applied on the functions $(f^{in})^\frac{1}{2}\abs{V_1(t)-V_2(t)}$ and $(f^{in})^\frac{1}{2}\abs{X_1(t)-X_2(t)}$. 
\begin{align*}
S(t)& \leq \norme{B}_{L^\infty\left( W^{1,\infty}\right)}\left(\int_{0}^{t}\norme{E_2(s)}_{L^\infty(\mathbb{T}^d)}e^{(t-s)\norme{B}_{\infty}}ds\right)\int_{(\mathbb{T}^d \times \R^d)^2} \abs{V_1-V_2}  \abs{X_1-X_2}^{1-\frac{1}{p}}d\pi_0,\\
&\leq  C \norme{B}_{L^\infty\left( W^{1,\infty}\right)}\left(\int_{0}^{t}\left(1+\norme{\rho_2(s)}_{L^\infty(\mathbb{T}^d)}\right)e^{(t-s)\norme{B}_{\infty}}ds\right)\left(\int_{(\mathbb{T}^d \times \R^d)^2}\abs{V_1-V_2}^2 d\pi_0\right)^\frac{1}{2}\\
&\times \left(\int_{(\mathbb{T}^d \times \R^d)^2}\abs{X_1-X_2}^{2(1-\frac{1}{p})} d\pi_0\right)^\frac{1}{2}.
%&\leq  C T\norme{B}_{W^{1,\infty}}e^{T\norme{ B}_\infty}(1+A(t))\frac{Q(t)}{\lambda(t)^\frac{1}{2}}.
\end{align*}
Now we use the Jensen inequality and the definition of $Q(t)$ to write
\begin{align*}
S(t)& \leq C \norme{B}_{L^\infty\left( W^{1,\infty}\right)}\left(\int_{0}^{t}\left(1+\norme{\rho_2(s)}_{L^\infty(\mathbb{T}^d)}\right)e^{(t-s)\norme{B}_{\infty}}ds\right)\left(\int_{(\mathbb{T}^d \times \R^d)^2}\abs{V_1-V_2}^2 d\pi_0\right)^\frac{1}{2}\\
&\times \left(\int_{(\mathbb{T}^d \times \R^d)^2}\abs{X_1-X_2}^{2} d\pi_0\right)^{\frac{1}{2}(1-\frac{1}{p})}\\
& \leq C \norme{B}_{L^\infty\left( W^{1,\infty}\right)}\left(\int_{0}^{t}\left(1+\norme{\rho_2(s)}_{L^\infty(\mathbb{T}^d)}\right)e^{(t-s)\norme{B}_{\infty}}ds\right) Q(t)^\frac{1}{2} \left(\frac{Q(t)}{\lambda(t)}\right)^{\frac{1}{2}(1-\frac{1}{p})}.
\end{align*} 

Now we turn to the control of $R(t)$. Like for the estimate on $S(t)$, we first use \eqref{Bholder} for $\alpha=1-\frac{1}{p}$ with $p>1$ to obtain

%Since $B \in L^\infty\left(\intervalleff{0}{T}, W^{1,\infty}(\mathbb{T}^3)\right)$ then for all $t \in \intervalleff{0}{T}$ and $\alpha \in \intervalleof{0}{1}$ 
%\begin{equation}\label{BinHolder}
%B(t) \in C^{0,\alpha}(\mathbb{T}^3)
%\end{equation}
%with H\"older coefficient $C_{B(t)}$ verifying $C_{B(t)} \leq \max(2\norme{B(t)}_\infty,\norme{\nabla B(t)}_\infty) \leq 2 \norme{B(t)}_{W^{1,\infty}}\leq 2 \norme{B}_{L^\infty\left(\intervalleff{0}{T}, W^{1,\infty}(\mathbb{T}^3)\right)}$.

%Thus for all $r>p$, we have that $B(t) \in C^{0,1-\frac{p}{r}}(\mathbb{T}^3)$, which allows us to bound $R(t)$
\begin{equation*}
R(t)\leq \norme{B}_{L^\infty\left( W^{1,\infty}\right)}e^{t\norme{ B}_\infty} \int_{(\mathbb{T}^d \times \R^d)^2} \abs{w}\abs{V_1-V_2}\abs{X_1-X_2}^{1-\frac{1}{p}}d\pi_0.
\end{equation*}
Now we apply the H\"older inequality for three functions with exponents $2p,2,\frac{2}{1-\frac{1}{p}}$ and use the definition of $Q(t)$ to obtain
\begin{align*}
R(t)&\leq \norme{B}_{L^\infty\left( W^{1,\infty}\right)}e^{t\norme{ B}_\infty}\left(\int_{(\mathbb{T}^d \times \R^d)^2}\abs{w}^{2p} d\pi_0\right)^\frac{1}{2p}\left(\int_{(\mathbb{T}^d \times \R^d)^2}\abs{V_1-V_2}^2 d\pi_0\right)^\frac{1}{2}\\
& \times \left(\int_{(\mathbb{T}^d \times \R^d)^2}\abs{X_1-X_2}^2 d\pi_0\right)^{\frac{1}{2}(1-\frac{1}{p})},\\
& \leq \norme{B}_{L^\infty\left( W^{1,\infty}\right)}e^{t\norme{ B}_\infty}\left(\int_{(\mathbb{T}^d \times \R^d)^2}\abs{w}^{2p} d\pi_0\right)^\frac{1}{2p} Q(t)^\frac{1}{2} \left(\frac{Q(t)}{\lambda(t)}\right)^{\frac{1}{2}(1-\frac{1}{p})},\\
& \leq \norme{B}_{L^\infty\left( W^{1,\infty}\right)}e^{t\norme{ B}_\infty}\left(\int_{\mathbb{T}^d \times \R^d}\abs{w}^{2p} df_2^0(y,w)\right)^\frac{1}{2p} Q(t)^{1-\frac{1}{2p}}\left(\frac{1}{\lambda(t)}\right)^{\frac{1}{2}(1-\frac{1}{p})}.
\end{align*}
The last inequality is obtained because we have the identity $\int_{(\mathbb{T}^d \times \R^d)^2}\abs{w}^r d\pi_0=\int_{\mathbb{T}^d \times \R^d}\abs{w}^r df_2^0(y,w)$.

Finally, using the assumption \eqref{condi:Orli}, we have
\begin{equation*}\label{ineq:R}
R(t) \leq C\norme{B}_{L^\infty\left( W^{1,\infty}\right)}e^{t\norme{ B}_\infty} p Q(t)^{1-\frac{1}{2p}}\left(\frac{1}{\lambda(t)}\right)^{\frac{1}{2}(1-\frac{1}{p})}.
\end{equation*}

We know differentiate between the cases $\lambda(t)=1$ and $\lambda(t)=\abs{\log Q(t)}$ to prove the two statements in \cref{theo:main}.

\subsubsection*{1st case with $\mathbf{\lambda(t)=1}$:}

From the above analysis, we deduce that the magnetized contribution to the derivative of $Q$ verifies for all $p >1$
\begin{equation*}
R(t)+S(t) \leq C\norme{B}_{L^\infty\left( W^{1,\infty}\right)}\left[p e^{t\norme{ B}_\infty} +\int_{0}^{t}\left(1+\norme{\rho_2(s)}_{L^\infty(\mathbb{T}^d)}\right)e^{(t-s)\norme{B}_{\infty}}ds\right]Q(t)^{1-\frac{1}{2p}}.
\end{equation*}
We will be working in the regime $Q(t) < 1/e$ which implies that $\abs{\log Q(t)}> 1$. This means we can impose the substitution $p= \abs{\log  Q(t)}$, which gives us $(Q(t))^{-\frac{1}{2p}}=\exp\left(-\frac{1}{2\abs{\log  Q(t)} } \log Q(t)\right)=e^2$ which finally implies that
\begin{equation}\label{ineq:finalBloeper}
R(t)+S(t) \leq C\norme{B}_{L^\infty\left( W^{1,\infty}\right)}\left[\abs{\log  Q(t)} e^{t\norme{ B}_\infty} +\int_{0}^{t}\left(1+\norme{\rho_2(s)}_{L^\infty(\mathbb{T}^d)}\right)e^{(t-s)\norme{B}_{\infty}}ds\right]Q(t).
\end{equation}
This is a satisfactory bound that will allow us to close the Gr\"onwall inequality. Furthermore, we see that the term $C \norme{B}_{L^\infty\left( W^{1,\infty}\right)} e^{t\norme{ B}_\infty}\abs{\log  Q(t)}Q(t)$ is of the same order as the term due to the electric field. Indeed, using the arguments from \cite{L06,HKI17} we have that
\begin{align*}
&\int_{(\mathbb{T}^d \times \R^d)^2} (X_1-X_2) \cdot  (V_1-V_2)d\pi_0
 + \int_{(\mathbb{T}^d \times \R^d)^2} (V_1-V_2) \cdot (E_1(t,X_1)-E_2(t,X_2))d\pi_0,\\
& \leq CA(t)\left(Q(t)+\sqrt{Q(t)}\sqrt{\varphi(Q(t))}\right),
\end{align*}
with 
\begin{equation*}
\varphi(t)\coloneqq
\left\{
\begin{aligned}
& t\log^2 t &\mbox{ for }& s \in \intervalleof{0}{1/e},\\
& t &\mbox{ for }& s >1/e.
\end{aligned}
\right.
\end{equation*}
Since we're in the regime $Q(t) < 1/e$, this implies that the electric term is bounded by a term of order $A(t)Q(t)\abs{\log  Q(t)}$.

Hence, we deduce the following final Gr\"onwall inequality for $Q$
\begin{equation*}
Q'(t) \leq c_d\left[A(t)+\norme{B}_{L^\infty\left( W^{1,\infty}\right)}\left( e^{t\norme{ B}_\infty} +\int_{0}^{t}\left(1+\norme{\rho_2(s)}_{L^\infty(\mathbb{T}^d)}\right)e^{(t-s)\norme{B}_{\infty}}ds\right)\right]Q(t)\abs{\log  Q(t)},
\end{equation*}
where $c_d$ is a dimensional constant.

This implies that in the regime $Q(t) < 1/e$ we have
\begin{equation*}
\begin{aligned}
&Q(t) \leq\\
&\exp\left[\log\left(Q(0)\right)\exp\left(-c_d \int_{0}^{t} \left[A(s)+\norme{B}_{L^\infty\left( W^{1,\infty}\right)}\left( e^{s\norme{ B}_\infty} +\int_{0}^{s}\left(1+\norme{\rho_2(u)}_{L^\infty(\mathbb{T}^d)}\right)e^{(s-u)\norme{B}_{\infty}}du\right)\right]ds\right)\right].
\end{aligned}
\end{equation*}
Since $Q(0)=W_2^2(f_1(0),f_2(0))$, the regime $Q(t) < 1/e$ is assured thanks to \eqref{condi:W2ini} and the condition $W_2^2(f_1(0),f_2(0)) < 1/e^2$. Finally we recall the classical inequality $W_2^2(f_1(t),f_2(t)) \leq Q(t)$ for all $t \in \intervalleff{0}{T}$. This concludes the proof of statement 1 of \cref{theo:main}.

\subsubsection*{2nd case with $\lambda(t)=\abs{\log Q(t)}$:}

Thanks to the above analysis we obtain for all $p>1$
\begin{equation*}
R(t)+S(t) \leq C\norme{B}_{L^\infty\left( W^{1,\infty}\right)}\left[p e^{t\norme{ B}_\infty} +\int_{0}^{t}\left(1+\norme{\rho_2(s)}_{L^\infty(\mathbb{T}^d)}\right)e^{(t-s)\norme{B}_{\infty}}ds\right]Q(t)^\frac{1}{2} \left(\frac{Q(t)}{\abs{\log Q(t)}}\right)^{\frac{1}{2}(1-\frac{1}{p})}.
\end{equation*}

We will also be working in the regime $Q(t) < 1/e$, but here we impose the substitution $p=\abs{\log \frac{Q(t)}{\abs{\log Q(t)}}}$. We are allowed to do this because since we're in the regime $Q(t) < 1/e$, then $\abs{\log Q(t)} < 1$ which implies $\frac{Q(t)}{\abs{\log Q(t)}} < Q(t) < 1/e$ and so $\abs{\log \frac{Q(t)}{\abs{\log Q(t)}}} >1$.
With this substitution, we have $\left(\frac{Q(t)}{\abs{\log Q(t)}}\right)^{-\frac{1}{2p}}=\exp\left(-\frac{1}{2\abs{\log  \frac{Q(t)}{\abs{\log Q(t)}}} } \log \frac{Q(t)}{\abs{\log Q(t)}}\right)=e^2$ which implies that
\begin{equation*}
R(t)+S(t) \leq C\norme{B}_{L^\infty\left( W^{1,\infty}\right)}\left[\abs{\log  \frac{Q(t)}{\abs{\log Q(t)}}} e^{t\norme{ B}_\infty} +\int_{0}^{t}\left(1+\norme{\rho_2(s)}_{L^\infty(\mathbb{T}^d)}\right)e^{(t-s)\norme{B}_{\infty}}ds\right]\frac{Q(t)}{\sqrt{\abs{\log Q(t)}}}.
\end{equation*}
We now observe that as long as $Q(t) \leq 1$ then
\begin{equation*}
\abs{\log  \frac{Q(t)}{\abs{\log Q(t)}}} \leq C\abs{\log Q(t)}.
\end{equation*}
This gives us the final estimate on the magnetized terms
\begin{equation}\label{ineq:finalBiaco}
R(t)+S(t) \leq C\norme{B}_{L^\infty\left( W^{1,\infty}\right)}\left[\abs{\log Q(t)} e^{t\norme{ B}_\infty} +\int_{0}^{t}\left(1+\norme{\rho_2(s)}_{L^\infty(\mathbb{T}^d)}\right)e^{(t-s)\norme{B}_{\infty}}ds\right]\frac{Q(t)}{\sqrt{\abs{\log Q(t)}}}.
\end{equation}
This bound will allow us to close the Gr\"onwall inequality. We also see the improvement compared to \eqref{ineq:finalBloeper}, because in \eqref{ineq:finalBiaco} the leading order term in $Q$ is $\sqrt{\abs{\log Q(t)}} Q(t)$, whereas in \eqref{ineq:finalBloeper} it is $\abs{\log Q(t)} Q(t)$. Furthermore, this term is of the same order in $Q$ as in the estimate on the electric terms obtained in \cite{I22}.

Indeed, the analysis in the proof of theorem 3.1 in \cite{I22} allows us to write
\begin{align*}
&\int_{(\mathbb{T}^d \times \R^d)^2} (X_1-X_2) \cdot  (V_1-V_2)d\pi_0
+ \int_{(\mathbb{T}^d \times \R^d)^2} (V_1-V_2) \cdot (E_1(t,X_1)-E_2(t,X_2))d\pi_0,\\
& \leq CA(t)\sqrt{\abs{\log Q(t)}} Q(t),
\end{align*}
as long as $Q(t) \leq 1/e$.

Hence, we deduce the following final Gr\"onwall inequality for $Q$
\begin{equation*}
Q'(t) \leq C_d\left[A(t)+\norme{B}_{L^\infty\left( W^{1,\infty}\right)}\left( e^{t\norme{ B}_\infty} +\int_{0}^{t}\left(1+\norme{\rho_2(s)}_{L^\infty(\mathbb{T}^d)}\right)e^{(t-s)\norme{B}_{\infty}}ds\right)\right]Q(t)\sqrt{\abs{\log Q(t)}},
\end{equation*}
where $C_d$ is a dimensional constant.

This implies that in the regime $Q(t) < 1/e$ we have
\begin{equation}\label{ineq:Qbis}
Q(t) \leq
e^{-\left[\sqrt{\log\left(Q(0)\right)}-C_d \int_{0}^{t} \left[A(s)+\norme{B}_{L^\infty\left( W^{1,\infty}\right)}\left( e^{s\norme{ B}_\infty} +\int_{0}^{s}\left(1+\norme{\rho_2(u)}_{L^\infty(\mathbb{T}^d)}\right)e^{(s-u)\norme{B}_{\infty}}du\right)\right]ds\right]^2.}
\end{equation}
From \eqref{ineq:Qbis} we see that the regime $Q(t) < 1/e$ is guaranteed if 
\begin{equation*}
\sqrt{\log\left(Q(0)\right)}\geq C_d \int_{0}^{t} \left[A(s)+\norme{B}_{L^\infty\left( W^{1,\infty}\right)}\left( e^{s\norme{ B}_\infty} +\int_{0}^{s}\left(1+\norme{\rho_2(u)}_{L^\infty(\mathbb{T}^d)}\right)e^{(s-u)\norme{B}_{\infty}}du\right)\right]ds+1.
\end{equation*}
Using the comparison between $Q(t)$ and $W_2(f_1(t),f_2(t))$ carried out in the proof of theorem 3.1 in \cite{I22}, and lemma 3.7 in \cite{I22} where the quantity $Q$ is shown to be well-defined (the proof is the same for \eqref{sys:VPwB}), we are able to conclude the proof of statement 2 of \cref{theo:main}.
%\begin{equation*}
%\frac{1}{2} W_2^2(f_1(t),f_2(t)) \leq Q(t),
%\end{equation*}
%and
%\begin{equation*}
%\frac{Q(0)}{\abs{\log Q(0)}} \leq \frac{1}{2} W_2^2(f_1(0),f_2(0)).
%\end{equation*}
%Proceeding as 

\subsection{Stability estimates when $B$ is log-lipschitz in position}\label{sec:Blogl}
In this section, we slightly weaken \eqref{reguB} and assume that the magnetic field is no longer Lipschitz but just log-lipschitz in position \eqref{ineq:Blogl}. More precisely, we assume that $B$ verifies \eqref{ineq:Blogl} with $\beta=1$:
\begin{equation}\label{reguBlogl}
\abs{B(t,x)-B(t,y)} \leq C \abs{x-y}\abs{\log\abs{x-y}}.
\end{equation}
In particular, \eqref{reguBlogl} implies that $B \in  L^\infty_{loc}\left(\R^+, L^\infty(\mathbb{T}^d)\right)$.
For such magnetic fields, we can adapt the proof of \cref{theo:main} in \cref{sec:VPwBproof} to obtain stability estimates for \eqref{sys:VPwB}. More precisely, we have to adapt the proof of the second statement of \cref{theo:main}, because for such a singular magnetic field we need to use the kinetic Wasserstein approach developed in \cite{I22}. This allows us to show the following result:
\begin{theorem}\label{theo:Blogl}
	Let $f_1,f_2,\rho_1, \rho_2, A$ be defined like in \cref{theo:main}, and we also define 
	\begin{equation*}
	K(s)=A(s)+\left( e^{s\norme{ B}_\infty} +\int_{0}^{s}\left(1+\norme{\rho_2(u)}_{L^\infty(\mathbb{T}^d)}\right)e^{(s-u)\norme{B}_{\infty}}du\right).
	\end{equation*}
For $T>0$, assume that $B$ verifies \eqref{reguBlogl} and that $A$ satisfies \eqref{condi:A}.\\
Assume also that there exists a universal constant $C_0$ such that for all $k\geq 1$
\begin{equation}\label{condi:Orlibis}
\int_{\mathbb{T}^d \times \R^d}\abs{v}^k df_2^0(x,v) \leq \sqrt{(C_0 k)^k}.
\end{equation}	
Then there exists a dimensional constant $C_d$ and a universal constant $c_0$ such that if $W_2^2(f_1(0),f_2(0))$ is sufficiently small so that $W_2^2(f_1(0),f_2(0)) < c_0$ and
	\begin{equation}\label{condi:W2initer}
	\abs{\log\left(W_2^2(f_1(0),f_2(0))\abs{\log\left(\frac{1}{2}W_2^2(f_1(0),f_2(0))\right)}\right)} \geq \exp\left(C_d \int_{0}^{t} K(s) ds\right),
	\end{equation}
	then
	\begin{equation}
	W_2^2(f_1(t),f_2(t)) \leq
	\exp\left[\log\left(W_2^2(f_1(0),f_2(0))\abs{\log\left(\frac{1}{2}W_2^2(f_1(0),f_2(0))\right)}\right)\exp\left(-C_d \int_{0}^{t} K(s) ds\right)\right].
	\end{equation}
\end{theorem}
This result is obtained by adapting the proof of \cref{theo:main}, we explain how below.
\begin{proof}
The main difference with the previous section is that, since $B$ is log-lipschitz, its H\"older norm is going to depend on $\alpha$ for $\alpha \in \intervalleoo{0}{1}$. Indeed if $B$ verifies \eqref{reguBlogl} then we have for $\alpha \in \intervalleoo{0}{1}$
\begin{equation}
\norme{B}_{C^{0,\alpha}(\T^d)} \leq C \underset{x,y \in \T^d}{\max}\left(\abs{x-y}^{1-\alpha}\abs{\log\abs{x-y}}\right).
\end{equation}
By straightforward computations, we show that the function $h \colon x \mapsto x^{1-\alpha} \abs{\log \abs{x}}$ on $\T^d$ reaches its maximum at $x_0=e^{-\frac{1}{1-\alpha}}$ and that $h(x_0)=\frac{e^{-1}}{1-\alpha}$.

Taking this into account in the proof presented in \cref{sec:VPwBproof}, we can write the following estimates on the magnetized terms $R(t)$ and $S(t)$ defined in the previous section
\begin{align*}
R(t)\leq C p e^{t\norme{ B}_\infty}\left(\int_{(\mathbb{T}^d \times \R^d)^2}\abs{w}^{2p} d\pi_0\right)^\frac{1}{2p} Q(t)^\frac{1}{2} \left(\frac{Q(t)}{\lambda(t)}\right)^{\frac{1}{2}(1-\frac{1}{p})},
\end{align*}
and 
\begin{align*}
S(t)\leq C p\left(\int_{0}^{t}\left(1+\norme{\rho_2(s)}_{L^\infty(\mathbb{T}^d)}\right)e^{(t-s)\norme{B}_{\infty}}ds\right) Q(t)^\frac{1}{2} \left(\frac{Q(t)}{\lambda(t)}\right)^{\frac{1}{2}(1-\frac{1}{p})},
\end{align*}
where $p=\frac{1}{1-\alpha}$ like in \cref{sec:VPwBproof}.

Using the new decay on the velocity moments \eqref{condi:Orlibis}, we obtain the following bounds on the magnetized terms
\begin{equation*}
	R(t)+S(t) \leq C \left(p^{\frac{3}{2}}e^{t\norme{ B}_\infty}+p\left(\int_{0}^{t}\left(1+\norme{\rho_2(s)}_{L^\infty(\mathbb{T}^d)}\right)e^{(t-s)\norme{B}_{\infty}}ds\right)\right)Q(t)^\frac{1}{2} \left(\frac{Q(t)}{\lambda(t)}\right)^{\frac{1}{2}(1-\frac{1}{p})}.
\end{equation*}
Now we proceed exactly like in the second case in \cref{sec:VPwBproof} where $\lambda(t)=\abs{\log Q(t)}$ and we impose the substitution $p=\abs{\log \frac{Q(t)}{\abs{\log Q(t)}}}$. Using the same analysis, we can deduce that in the regime $Q(t) \leq 1/e$ the following estimate holds
\begin{equation*}
Q'(t) \leq C_d\left[A(t)+\left( e^{t\norme{ B}_\infty} +\int_{0}^{t}\left(1+\norme{\rho_2(s)}_{L^\infty(\mathbb{T}^d)}\right)e^{(t-s)\norme{B}_{\infty}}ds\right)\right]Q(t)\abs{\log Q(t)},
\end{equation*}
where $C_d$ is a dimensional constant.

This implies that in the regime $Q(t) \leq 1/e$ we have
\begin{equation*}
\begin{aligned}
&Q(t) \leq\\
&\exp\left[\log\left(Q(0)\right)\exp\left(-C_d \int_{0}^{t} \left[A(s)+\left( e^{s\norme{ B}_\infty} +\int_{0}^{s}\left(1+\norme{\rho_2(u)}_{L^\infty(\mathbb{T}^d)}\right)e^{(s-u)\norme{B}_{\infty}}du\right)\right]ds\right)\right].
\end{aligned}
\end{equation*}
Like for the proof of the second statement of \cref{theo:main}, we use the comparison between $Q(t)$ and $W_2(f_1(t),f_2(t))$ carried out in the proof of theorem 3.1 in \cite{I22}, and lemma 3.7 in \cite{I22} where the quantity $Q$ is shown to be well-defined (the proof is the same for \eqref{sys:VPwB}). This allows us to conclude the proof of \cref{theo:Blogl}.

\end{proof}

\section{An improved Dobrushin estimate for the magnetized Vlasov equation} \label{sec:V}

We consider the magnetized Vlasov equation with smooth kernel given by
\begin{equation}\label{eq:Vlasov}\tag{VB}
\partial_t f + v\cdot \nabla_x f + \left(F[f]+v \wedge B\right) \cdot \nabla_{v} f= 0,
\end{equation}
with $F[f]\coloneqq \nabla (K \ast \rho_f)$, $\rho_f \coloneqq \int f dv$, and the initial condition $f^{in} \coloneqq f(0)$ which is a probability density. Here $K$ is a $C^{1,1}$ potential such that $\norme{D^2K}_\infty \eqqcolon H < +\infty$ and the external magnetic field $B$ is constant and given by $B \coloneqq (0,0,\omega)$, with $\omega \geq 0$. Indeed this choice of $B$ covers all possible constant magnetic fields because the Vlasov equation is invariant under rotations.

\subsection{Main result and preliminaries}\label{sec:VwBmain}

The results in this section will be valid for both dimension $d=2$ and $d=3$, but as explained in the introduction the results will be quite different in both dimensions.
We are going to prove the following theorem.
\begin{theorem}\label{theo:2}
	Let $d=2,3$ and let $f_1,f_2$ be two solutions of \eqref{eq:Vlasov}.
	
	If $d=2$, then 
	\begin{equation}
	\begin{aligned}
	& W_1(f_1(t),f_2(t)) \leq \\
	&\min\left(\left(\sqrt{\frac{2(1-\cos(\omega t))}{\omega^2}}+1\right)e^{4H\left(\frac{2(t-\frac{\sin(\omega t)}{\omega})}{\omega^2}+t\right)},e^{(1+2H)t}\right)W_1(f_1(0),f_2(0)).
	\end{aligned}
	\end{equation}
	
	If $d=3$, then 
	\begin{equation}
	\begin{aligned}
	&W_1(f_1(t),f_2(t)) \leq\\
	&\min\left(\left(\sqrt{\frac{2(1-\cos(\omega t))}{\omega^2}+t^2}+1\right)e^{4H\left(\frac{2(t-\frac{\sin(\omega t)}{\omega})}{\omega^2}+\frac{t^3}{3}+t\right)},e^{(1+2H)t}\right)W_1(f_1(0),f_2(0)).
	\end{aligned}
	\end{equation}
\end{theorem}

Now we give three important remarks.

\begin{remark}
	Notice that we recover the same type of estimate as in \cite[theorem 2.1]{I22} when $\omega \to 0$ because $\frac{2(1-\cos(\omega t))}{\omega^2} \underset{\omega \to 0}{\longrightarrow} t^2$ and $\frac{2(t-\frac{\sin(\omega t)}{\omega})}{\omega^2} \underset{\omega \to 0}{\longrightarrow} \frac{t^3}{3}$.
\end{remark}
\begin{remark}
	The stability estimates obtained in \cref{theo:2} for the $1$-Wasserstein metric can be easily generalized to all $p$-Wasserstein metric for $1<p<+\infty$.
\end{remark}
\begin{remark}
	A Dobrushin type estimate can also be written if $K$ is a potential with only log-Lipschitz regularity. In this case, we obtain a log-Lipschitz differential inequality in the proof instead of a classical Gr\"onwall inequality. Thus, the improvement stated in \cref{theo:2} can also be generalized to this situation.
\end{remark}
Before giving the proof, we introduce some useful quantities and discuss the relevance of \cref{theo:2}.

As usual for the Vlasov equation, the characteristics of \eqref{eq:Vlasov} are defined by
\begin{equation*}\label{def:charaVlasov}
\left\{
\begin{aligned}
& \frac{d}{ds}X(s;t,x,v)=V(s;t,x,v),\\
& \frac{d}{ds}V(s;t,x,v)=\nabla(K \ast \rho_f)(t,X(s;t,x,v))+V(s;t,x,v) \wedge B,
\end{aligned}
\right.
\end{equation*}
with 
\begin{equation*}
\left(X(t;t,x,v),V(t;t,x,v)\right)=(x,v).
\end{equation*}
We call \textit{the magnetized free-transport equation} the Vlasov equation without the nonlinear term $F[f]\cdot \nabla_{v} f$
\begin{equation}\label{eq:ftB}
\partial_t f + v\cdot \nabla_x f + \left(v \wedge B\right) \cdot \nabla_{v} f= 0,
\end{equation}
and we also write its associated characteristics
\begin{equation}\label{def:charaftB}
\left\{
\begin{aligned}
& \frac{d}{ds}X_\omega(s;t,x,v)=V_\omega(s;t,x,v),\\
& \frac{d}{ds}V_\omega(s;t,x,v)=V_\omega(s;t,x,v) \wedge B,
\end{aligned}
\right.
\end{equation}
with 
\begin{equation*}
\left(X_\omega(t;t,x,v),V_\omega(t;t,x,v)\right)=(x,v).
\end{equation*}
Since the external magnetic field is constant, we can solve the above ordinary differential equation explicitly to obtain the following expression for the characteristics of \eqref{eq:ftB}.

For $d=2$, we have
\begin{equation*}\label{eq:charaftB32}
\left\{
\begin{aligned}
& V_\omega(s;t,x,v)=  \begin{pmatrix}
v_1\cos(\omega (s-t))+v_2\sin(\omega (s-t))\\
-v_1\sin(\omega (s-t))+v_2\cos(\omega (s-t))
\end{pmatrix},\\
& X_\omega(s;t,x,v)=\begin{pmatrix}
x_1+\frac{v_1}{\omega}\sin(\omega (s-t))+\frac{v_2}{\omega}(1-\cos(\omega (s-t)))\\
x_2+\frac{v_1}{\omega}(\cos(\omega (s-t))-1)+\frac{v_2}{\omega}\sin(\omega (s-t))
\end{pmatrix}.
\end{aligned}
\right.
\end{equation*}

For $d=3$, we have
\begin{equation*}\label{eq:charaftB3}
\left\{
\begin{aligned}
& V_\omega(s;t,x,v)=  \begin{pmatrix}
v_1\cos(\omega (s-t))+v_2\sin(\omega (s-t))\\
-v_1\sin(\omega (s-t))+v_2\cos(\omega (s-t))\\
v_3
\end{pmatrix},\\
& X_\omega(s;t,x,v)=\begin{pmatrix}
x_1+\frac{v_1}{\omega}\sin(\omega (s-t))+\frac{v_2}{\omega}(1-\cos(\omega (s-t)))\\
x_2+\frac{v_1}{\omega}(\cos(\omega (s-t))-1)+\frac{v_2}{\omega}\sin(\omega (s-t))\\
x_3+v_3(s-t)
\end{pmatrix}.
\end{aligned}
\right.
\end{equation*}

Thus we can write the solution of $f^\omega$ \eqref{eq:ftB}
\begin{equation}\label{eq:solftB}
 f^\omega(t,x,v)=f^{in}(X_\omega(0;t,x,v),V_\omega(0;t,x,v)).
\end{equation}

As stated above, the estimates in \cref{theo:2} are optimal in the regime when $H$ is small, contrary to the classical Dobrushin estimate. Indeed in the regime  $\norme{D^2K}_\infty \eqqcolon H \ll 1$, we would like that our estimate encodes the fact that the solutions are close to \eqref{eq:solftB} (with different initial datum). In terms of the $1$-Wasserstein distance, this just translates to  $W_1(f_1(t),f_2(t)) \approx W_1(f_1^\omega(t),f_2^\omega(t))$ with $f_1^\omega$ given by \eqref{eq:solftB} with initial data $f_1^\omega(0)$, and $f_2^\omega$ given by \eqref{eq:solftB} with initial data $f_2^\omega(0)$.  From the expression of the characteristics of the magnetized free-transport equation $X_\omega, V_\omega$, we observe that for $d=3$, $W_1(f_1^\omega(t),f_2^\omega(t)) = O(t)$, and for $d=2$, $W_1(f_1^\omega(t),f_2^\omega(t)) = O(1)$. This leads us to give the following key remark.

\begin{remark}
	We notice that the above bounds are optimal when $H=0$. Indeed, in this regime \cref{theo:2} gives us that $W_1(f_1(t),f_2(t)) \leq \frac{2}{\omega}+1=O(1)$ for $d=2$, and $W_1(f_1(t),f_2(t)) \leq t+\frac{2}{\omega}+1=O(t)$ for $d=3$, which are the expected bounds on $W_1(f_1(t),f_2(t))$. It also provides a better estimate than the usual Dobrushin bound in the small $H$ and small $t$ regime.
\end{remark}
Before giving the proof, we quickly recall how the classical Dobrushin estimate is obtained to show the improvement made in \cref{theo:2}. We also highlight how this estimate doesn't change with the dimension and is the same in both magnetized and unmagnetized frameworks.

\begin{theorem}[Dobrushin stability estimate]\label{theo:Dobru}
Let $f_1,f_2$ be two solutions to \eqref{eq:Vlasov} with either $d=2$ or $d=3$. Then the 1-Wasserstein distance $W_1(f_1(t),f_2(t))$ between $f_1$ and $f_2$ satisfies
\begin{equation}\label{ineq:Dob}
W_1(f_1(t),f_2(t) \leq e^{(1+2H)t} W_1(f_1(0),f_2(0).
\end{equation}
\end{theorem}

%We show very briefly what changes in the magnetized case to obtain the above estimate.
\begin{proof}
As usual, we consider $\pi_0$ an optimal $W_1$ coupling between $f_1(0)$ and $f_2(0)$ so that we can bound $W_1(f_1(t),f_2(t))$ by a quantity $N(t)$ depending on the characteristics of \eqref{eq:Vlasov}
\begin{align*}
N(t)\coloneqq \int_{(\mathbb{T}^d \times \R^d)^2} \abs{X_1(t,x,v)-X_2(t,y,w)}+\abs{V_1(t,x,v)-V_2(t,y,w)} d\pi_0(x,v,y,w).
\end{align*}

By definition of the $1$-Wasserstein distance we have $W_1(f_1(t),f_2(t) \leq N(t)$. Differentiating $N$, we make the basic observation
\begin{align*}
&\frac{d}{dt}\int_{(\mathbb{T}^d \times \R^d)^2} \abs{V_1(t)-V_2(t)}d\pi_0 = \int_{(\mathbb{T}^d \times \R^d)^2} \frac{\left(V_1-V_2\right)\cdot \left(\dot{V}_1-\dot{V}_2\right)}{\abs{V_1-V_2}}d\pi_0\\
&= \int_{(\mathbb{T}^d \times \R^d)^2} \frac{\left(V_1-V_2\right)\cdot \left(\nabla (K \ast \rho_{f_1})(t,X_1)-\nabla (K \ast \rho_{f_2})(t,X_2) +(V_1-V_2)\wedge B\right)}{\abs{V_1-V_2}}d\pi_0\\
&=\int_{(\mathbb{T}^d \times \R^d)^2} \frac{\left(V_1-V_2\right)\cdot \left(\nabla (K \ast \rho_{f_1})(t,X_1)-\nabla (K \ast \rho_{f_2})(t,X_2)\right)}{\abs{V_1-V_2}}d\pi_0.
\end{align*}
So in the end we can bound the derivative of $N$ just like in the unmagnetized case
\begin{align*}
\frac{d}{dt} N(t) &\leq \int_{(\mathbb{T}^d \times \R^d)^2} \abs{\dot{X}_1-\dot{X}_2}+\abs{\dot{V}_1-\dot{V}_2} d\pi_0\\
%& = \int_{(\mathbb{T}^d \times \R^d)^2} \abs{V_1-V_2}+\abs{\nabla (K \ast \rho_{f_1})(t,X_1)-\nabla (K \ast \rho_{f_2})(t,X_2) +(V_1-V_2)\wedge B} d\pi_0\\
%&\leq \int_{(\mathbb{T}^d \times \R^d)^2} (1+\omega)\abs{V_1-V_2}+\abs{\nabla (K \ast \rho_{f_1})(t,X_1)-\nabla (K \ast \rho_{f_2})(t,X_2)} d\pi_0\\
& \leq N(t)+\int_{(\mathbb{T}^d \times \R^d)^2} \abs{\nabla (K \ast \rho_{f_1})(t,X_1)-\nabla (K \ast \rho_{f_2})(t,X_2)} d\pi_0.
\end{align*}
We can bound the second term in the inequality above using the standard argument by Dobrushin \cite{D79,G16} which gives us
\begin{equation}
\int_{(\mathbb{T}^d \times \R^d)^2} \abs{\nabla (K \ast \rho_{f_1})(t,X_1)-\nabla (K \ast \rho_{f_2})(t,X_2)} d\pi_0 \leq 2H N(t).
\end{equation}
The inequality \eqref{ineq:Dob} immediately follows from the above estimate, the inequality on $\frac{d}{dt} N(t)$ and the fact that the $1$-Wasserstein is bounded by $N$.% Furthermore we recover the usual Dobrushin stability estimate for the Vlasov equation when we taken $\omega=0$.
\end{proof}
%We give a last remark concerning the Dobrushin estimate before presenting the proof of \cref{theo:2}.
\begin{remark}
The Dobrushin estimate for solutions to \eqref{eq:Vlasov} isn't optimal in the small $H$ regime, because if $H=0$ then \cref{theo:Dobru} only tells us that $W_1(f_1(t),f_2(t) \leq O(e^t)$. As said above, in reality in this regime $W_1(f_1^\omega(t),f_2^\omega(t)) \approx O(1)$ for $d=2$, and $W_1(f_1(t),f_2(t) \approx O(t)$ for $d=3$.
\end{remark}

\subsection{Proof of \cref{theo:2}}\label{sec:VwBproof}

As usual we write $X_i,V_i$ the characteristics associated to $f_i$ and we consider $\pi_0$ an optimal $W_1$-coupling between $f_1(0)$ and $f_2(0)$. We define the quantity $Q$ defined for $t\geq 0$ by
\begin{equation*}\label{def:Q}
Q(t) \coloneqq \int_{(\mathbb{T}^d \times \R^d)^2} \left(\abs{X_\omega(0;t,X_1,V_1)-X_\omega(0;t,X_2,V_2)}+\abs{V_\omega(0;t,X_1,V_1)-V_\omega(0;t,X_2,V_2)}\right)d\pi_0.
\end{equation*}

Note that 
\begin{align*}
Q(0)&=\int_{(\mathbb{T}^d \times \R^d)^2}\left(\abs{X_\omega(0;0,x,v)-X_\omega(0;0,y,w)}+\abs{V_\omega(0;0,x,v)-V_\omega(0;0,y,w)}\right)d\pi_0(x,v,y,w)\\
&= \int_{(\mathbb{T}^d \times \R^d)^2}\left(\abs{x-y}+\abs{v-w}\right)d\pi_0(x,v,y,w)=W_1(f_1(0),f_2(0)).
\end{align*}
We also note that since $V_\omega(0;t,x,v)= R_\omega(t) v$ with
\begin{equation*}
R_\omega(t) \coloneqq
\left\{
\begin{aligned}
& 
\begin{pmatrix}
\cos(\omega t) & -\sin(\omega t)\\
\sin(\omega t) & \cos(\omega t)
\end{pmatrix} \quad \mbox{ for } d=2,\\
&  
\begin{pmatrix}
\cos(\omega t) & -\sin(\omega t) & 0\\
\sin(\omega t) & \cos(\omega t) &0\\
0 & 0 & 1
\end{pmatrix} \quad \mbox{ for } d=3,
\end{aligned}
\right.
\end{equation*}

the matrix of a rotation of angle $\omega t$. Then we have 
\begin{align*}
\abs{V_\omega(0;t,X_1,V_1)-V_\omega(0;t,X_2,V_2)}=\abs{R_\omega(t)(V_1-V_2)}=\abs{V_1-V_2},
\end{align*}
which of course means that
\begin{equation*}
Q(t) = \int_{(\mathbb{T}^d \times \R^d)^2} \left(\abs{X_\omega(0;t,X_1,V_1)-X_\omega(0;t,X_2,V_2)}+\abs{V_1-V_2}\right)d\pi_0.
\end{equation*}

Now we wish to find a Gr\"onwall type inequality for $Q$, and so we need to compute the derivatives of $X_\omega(0;t,X_i,V_i)$, which we can write in the following way
\begin{equation*}
X_\omega(0;t,X,V)=X+\frac{D_\omega(t)}{\omega}V,
\end{equation*}
with
\begin{equation*}
D_\omega(t) \coloneqq
\left\{
\begin{aligned}
& 
\begin{pmatrix}
-\sin(\omega t) & 1-\cos(\omega t)\\
\cos(\omega t)-1 & -\sin(\omega t)
\end{pmatrix} \quad \mbox{ for } d=2,\\
&  
\begin{pmatrix}
-\sin(\omega t) & 1-\cos(\omega t) & 0\\
\cos(\omega t)-1 & -\sin(\omega t) & 0\\
0 & 0 & -\omega t
\end{pmatrix} \quad \mbox{ for } d=3.
\end{aligned}
\right.
\end{equation*}

This means that
\begin{align*}
\frac{d}{dt}\left(X_\omega(0;t,X,V)\right)&=\dot{X}+\frac{\dot{D_\omega}(t)}{\omega}V+\frac{D_\omega(t)}{\omega}\dot{V}\\
&= V+\frac{\dot{D_\omega}(t)}{\omega}V+\frac{D_\omega(t)}{\omega}(F[f](t;0,X)+V \wedge B).
\end{align*}
Now we observe that

\begin{equation*}
\frac{\dot{D_\omega}(t)}{\omega}V \coloneqq
\left\{
\begin{aligned}
& 
\begin{pmatrix}
-\cos(\omega t)V_1+ \sin(\omega t)V_2\\
-\sin(\omega t)V_1-\cos(\omega t)V_2
\end{pmatrix} \quad \mbox{ for } d=2,\\
&  
\begin{pmatrix}
-\cos(\omega t)V_1+ \sin(\omega t)V_2\\
\sin(\omega t)V_1-\cos(\omega t)V_2 \\
-V_3
\end{pmatrix} \quad \mbox{ for } d=3,
\end{aligned}
\right.
\end{equation*}
and
\begin{equation*}
\frac{D_\omega(t)}{\omega}(V \wedge B) \coloneqq
\left\{
\begin{aligned}
& 
\begin{pmatrix}
(\cos(\omega t)-1)V_1- \sin(\omega t)V_2\\
\sin(\omega t)V_1+(\cos(\omega t)-1)V_2
\end{pmatrix} \quad \mbox{ for } d=2,\\
&  
\begin{pmatrix}
(\cos(\omega t)-1)V_1- \sin(\omega t)V_2\\
\sin(\omega t)V_1+(\cos(\omega t)-1)V_2 \\
0
\end{pmatrix} \quad \mbox{ for } d=3.
\end{aligned}
\right.
\end{equation*}
Finally we obtain 
\begin{equation*}
\frac{d}{dt}\left(X_\omega(0;t,X,V)\right)=\frac{D_\omega(t)}{\omega}F[f](t;0,X).
\end{equation*}
For any vector $x \in \mathbb{T}^d$, we have
\begin{equation}\label{eq:Domega}
\abs{\frac{D_\omega(t)}{\omega}x}=
\left\{
\begin{aligned}
&\sqrt{2(x_1^2+x_2^2)\frac{1-\cos(\omega t)}{\omega^2}} \quad \mbox{ for } d=2,\\
&\sqrt{2(x_1^2+x_2^2)\frac{1-\cos(\omega t)}{\omega^2}+t^2x_3^2} \quad \mbox{ for } d=3.
\end{aligned}
\right.
\end{equation}
\begin{remark}
Thanks to the previous equality and the expressions of $X_\omega,V_\omega$, we notice that we recover the same identity as in \cite{I22} when $\omega \to 0$, namely
\begin{equation*}
\abs{\frac{d}{dt}\left(X-tV)\right)}=t\abs{F[f]}.
\end{equation*}
\end{remark}
Thanks to \eqref{eq:Domega}, we can write
\begin{equation}\label{ineq:Domega}
\abs{\frac{D_\omega(t)}{\omega}x} \leq
\left\{
\begin{aligned}
&\sqrt{\frac{2(1-\cos(\omega t))}{\omega^2}}\abs{x} \quad \mbox{ for } d=2,\\
&\sqrt{\frac{2(1-\cos(\omega t))}{\omega^2}+t^2}\abs{x} \quad \mbox{ for } d=3.
\end{aligned}
\right.
\end{equation}
Before estimating $\dot{Q}$, we describe how $W_1(f_1(t),f_2(t))$ is controlled by $Q(t)$, we have
\begin{align*}
W_1(f_1(t),f_2(t)) &\leq \int_{(\mathbb{T}^d \times \R^d)^2} \abs{X_1(t,x,v)-X_2(t,y,w)}+\abs{V_1(t,x,v)-V_2(t,y,w)} d\pi_0(x,v,y,w)\\
& \leq \int_{(\mathbb{T}^d \times \R^d)^2} \left(\abs{X_\omega(0;t,X_1,V_1)-X_\omega(0;t,X_2,V_2)}+\abs{V_1-V_2}\right)d\pi_0\\
&+\int_{(\mathbb{T}^d \times \R^d)^2} \abs{X_\omega(0;t,X_1,V_1)-X_1-(X_\omega(0;t,X_2,V_2)-X_2)}d\pi_0\\
&= Q(t)+\int_{(\mathbb{T}^d \times \R^d)^2} \abs{\frac{D_\omega(t)}{\omega}(V_1-V_2)}d\pi_0.
\end{align*}
Thanks to \eqref{ineq:Domega}, we obtain the following estimate on $W_1(f_1(t),f_2(t))$

\begin{equation}\label{ineq:W1Q}
W_1(f_1(t),f_2(t)) \leq
\left\{
\begin{aligned}
&\left(\sqrt{\frac{2(1-\cos(\omega t))}{\omega^2}}+1\right)Q(t) \quad \mbox{ for } d=2,\\
&\left(\sqrt{\frac{2(1-\cos(\omega t))}{\omega^2}+t^2}+1\right)Q(t) \quad \mbox{ for } d=3.
\end{aligned}
\right.
\end{equation}

We can finally write and estimate the derivative of $Q$.
\begin{align*}
\frac{d}{dt} Q(t) &= \int_{(\mathbb{T}^d \times \R^d)^2} \left(\frac{\left(X_\omega(0;t,X_1,V_1)-X_\omega(0;t,X_2,V_2)\right)\cdot \left(\dot{X}_\omega(0;t,X_1,V_1)-\dot{X}_\omega(0;t,X_2,V_2)\right)}{\abs{X_\omega(0;t,X_1,V_1)-X_\omega(0;t,X_2,V_2)}}\right.\\
&+\left.\frac{\left(V_1-V_2\right)\cdot \left(\dot{V}_1-\dot{V}_2\right)}{\abs{V_1-V_2}}\right)d\pi_0\\
&= \int_{(\mathbb{T}^d \times \R^d)^2} \left(\frac{\left(X_\omega(0;t,X_1,V_1)-X_\omega(0;t,X_2,V_2)\right)\cdot \left(\frac{D_\omega(t)}{\omega}\left(F[f_1](t;0,X_1)-F[f_2](t;0,X_2)\right)\right)}{\abs{X_\omega(0;t,X_1,V_1)-X_\omega(0;t,X_2,V_2)}}\right.\\
&+\left.\frac{\left(V_1-V_2\right)\cdot \left(F[f_1](t;0,X_1)-F[f_2](t;0,X_2)+(V_1-V_2)\wedge B\right)}{\abs{V_1-V_2}}\right)d\pi_0\\
\end{align*}
We see the term $\left(V_1-V_2\right)\cdot(V_1-V_2)\wedge B=0$ appear in the second term in the above equality. This means we can finally estimate the derivative of $Q$ using only the interaction term $F[f_i](t;0,X_i)$ as was the goal.

Thus, using \eqref{ineq:Domega}, we obtain
\begin{equation*}
\frac{d}{dt} Q(t) \leq \left\{
\begin{aligned}
&\left(\sqrt{\frac{2(1-\cos(\omega t))}{\omega^2}}+1\right)\int_{(\mathbb{T}^d \times \R^d)^2}\abs{F[f_1](t;0,X_1)-F[f_2](t;0,X_2)}d\pi_0 \quad \mbox{ for } d=2,\\
&\left(\sqrt{\frac{2(1-\cos(\omega t))}{\omega^2}+t^2}+1\right)\int_{(\mathbb{T}^d \times \R^d)^2}\abs{F[f_1](t;0,X_1)-F[f_2](t;0,X_2)}d\pi_0 \quad \mbox{ for } d=3.
\end{aligned}
\right.
\end{equation*}
We bound the differences of the interaction fields using the same method as in \cite{I22}, using the fact that $\nabla K$ is $H$-Lipschitz, to obtain
\begin{equation*}
\int_{(\mathbb{T}^d \times \R^d)^2}\abs{F[f_1](t;0,X_1)-F[f_2](t;0,X_2)}d\pi_0 \leq 2H \int_{(\mathbb{T}^d \times \R^d)^2}\abs{X_1-X_2}d\pi_0.
\end{equation*}
Using \eqref{ineq:W1Q}, we get
\begin{equation*}
\frac{d}{dt} Q(t) \leq \left\{
\begin{aligned}
&2H\left(\sqrt{\frac{2(1-\cos(\omega t))}{\omega^2}}+1\right)^2 Q(t) \quad \mbox{ for } d=2,\\
&2H\left(\sqrt{\frac{2(1-\cos(\omega t))}{\omega^2}+t^2}+1\right)^2 Q(t) \quad \mbox{ for } d=3.
\end{aligned}
\right.
\end{equation*}
Using the simple inequality $(x+y)^2 \leq 2(x^2+y^2)$, we finally obtain the desired Gr\"onwall inequality on  $Q$
\begin{equation*}
\frac{d}{dt} Q(t) \leq \left\{
\begin{aligned}
&4H\left(\frac{2(1-\cos(\omega t))}{\omega^2}+1\right) Q(t) \quad \mbox{ for } d=2,\\
&4H\left(\frac{2(1-\cos(\omega t))}{\omega^2}+t^2+1\right) Q(t) \quad \mbox{ for } d=3.
\end{aligned}
\right.
\end{equation*}
Therefore we obtain the following bound on $Q$
\begin{equation*}
Q(t) \leq \left\{
\begin{aligned}
&e^{4H\left(\frac{2(t-\frac{\sin(\omega t)}{\omega})}{\omega^2}+t\right)} Q(0) \quad \mbox{ for } d=2,\\
&e^{4H\left(\frac{2(t-\frac{\sin(\omega t)}{\omega})}{\omega^2}+\frac{t^3}{3}+t\right)} Q(0) \quad \mbox{ for } d=3.
\end{aligned}
\right.
\end{equation*}
Finally, we get the expected estimate on the $1$-Wasserstein distance between $f_1(t)$ and $f_2(t)$.
\begin{equation*}
W_1(f_1(t),f_2(t)) \leq \left\{
\begin{aligned}
&\left(\sqrt{\frac{2(1-\cos(\omega t))}{\omega^2}}+1\right)e^{4H\left(\frac{2(t-\frac{\sin(\omega t)}{\omega})}{\omega^2}+t\right)} W_1(f_1(0),f_2(0)) \quad \mbox{ for } d=2,\\
&\left(\sqrt{\frac{2(1-\cos(\omega t))}{\omega^2}+t^2}+1\right)e^{4H\left(\frac{2(t-\frac{\sin(\omega t)}{\omega})}{\omega^2}+\frac{t^3}{3}+t\right)} W_1(f_1(0),f_2(0)) \quad \mbox{ for } d=3.
\end{aligned}
\right.
\end{equation*}

\textbf{Acknowledgments:} I would like to thank Mikaela Iacobelli for interesting and useful discussions, and the anonymous referee for comments which improved the quality of this work.
%\nocite{*}
\printbibliography

\end{document}